\documentclass[11pt]{article}

\usepackage[utf8]{inputenc}
\usepackage[margin=1in]{geometry}
\usepackage{amsmath,amsthm,amssymb}
\usepackage{comment}
\usepackage{enumerate}
\usepackage{xfrac}
\usepackage{tikz-cd}
\usepackage{accents} 
\usepackage[colorlinks=true,hidelinks]{hyperref}
    \AtBeginDocument{}
    
\usepackage{thmtools}
\usepackage{thm-restate}

\usepackage{tikz} 
\usetikzlibrary{arrows}
\usetikzlibrary{patterns}
\usepackage{pgfplots} 
\usetikzlibrary{decorations.markings} 
\usetikzlibrary{decorations.pathreplacing}
\usetikzlibrary{3d,calc} 
\usepackage{pifont}
\DeclareFontFamily{OT1}{pzc}{}
  \DeclareFontShape{OT1}{pzc}{m}{it}{<-> s * [1.200] pzcmi7t}{}
  \DeclareMathAlphabet{\mathpzc}{OT1}{pzc}{m}{it}
\usepackage{fancyhdr}
\newcommand{\Z}{\mathbb{Z}}

\renewcommand{\P}{\mathbb{P}}

\newcommand{\X}{\mathfrak{X}}

\renewcommand{\O}{\mathcal{O}}

\renewcommand{\P}{\mathbb{P}}

\DeclareFontFamily{U}{wncy}{}
    \DeclareFontShape{U}{wncy}{m}{n}{<->wncyr10}{}
    \DeclareSymbolFont{mcy}{U}{wncy}{m}{n}
    \DeclareMathSymbol{\Sha}{\mathord}{mcy}{"58} 

\newcommand{\ind}{\ensuremath{\operatorname{ind}}}

\newcommand{\Div}{\ensuremath{\operatorname{Div}}}
\newcommand{\Prin}{\ensuremath{\operatorname{Prin}}}

\newcommand{\Pic}{\ensuremath{\operatorname{Pic}}}

\newcommand{\PPic}{\ensuremath{\operatorname{\mathbf{Pic}}}}

\renewcommand{\epsilon}{\varepsilon}

\newtheorem{thm}{Theorem}[section]
\newtheorem{prop}[thm]{Proposition}

\newtheorem{lem}[thm]{Lemma}
\newtheorem*{defn}{Definition}
\newtheorem{cor}[thm]{Corollary}

\newtheorem*{rem}{Remark}
  \let\oldrem\rem
  \renewcommand{\rem}{\oldrem\normalfont}

\newtheorem{ex}[thm]{Example}
  \let\oldex\ex
  \renewcommand{\ex}{\oldex\normalfont}

\definecolor{nodecol}{RGB}{200,0,0}
\definecolor{compcol}{RGB}{0,0,200}

\def\fourcurvessetup{
    \def\pt{{\filldraw[compcol] (0,0) circle (2pt);}};
    \node (vt) at (0,2.5) {};
    \node (vtl) at (-.5,3.75) {};
    \node (vtr) at (.5,3.75) {};
    
    \node (vr) at (2.5,0) {};
    \node (vrt) at (3.75,.5) {};
    \node (vrb) at (3.75,-.5) {};
    
    \node (vb) at (0,-2.5) {};
    \node (vbl) at (-.5,-3.75) {};
    \node (vbr) at (.5,-3.75) {};
    
    \node (vl) at (-2.5,0) {};
    \node (vlt) at (-3.75,.5) {};
    \node (vlb) at (-3.75,-.5) {};
}

\def\fourcurvesfinish{
    \filldraw[nodecol] (vt) circle (4pt);
    \filldraw[nodecol] (vr) circle (4pt);
    \filldraw[nodecol] (vl) circle (4pt);
    \filldraw[nodecol] (vb) circle (4pt);
    \draw (vt) circle (4pt);
    \draw (vr) circle (4pt);
    \draw (vl) circle (4pt);
    \draw (vb) circle (4pt);
    
    \node (atl) at (-0.5, 2.1) {};
    \node (atr) at (0.5, 2.1) {};
    \draw[nodecol,dashed, postaction={decorate}, decoration={markings,mark=at position .65 with
        {\arrow[rotate=-11,scale=1.5,>=Stealth]{>}}}]  (atl) to[bend right] (atr);
        
    \node (art) at (2.1,.5) {};
    \node (arb) at (2.1,-.5) {};
    \draw[nodecol,dashed, postaction={decorate}, decoration={markings,mark=at position .65 with
        {\arrow[rotate=-11,scale=1.5,>=Stealth]{>}}}]  (art) to[bend right] (arb);
        
    \node (abl) at (-0.5, -2.1) {};
    \node (abr) at (0.5, -2.1) {};
    \draw[nodecol, dashed, postaction={decorate}, decoration={markings,mark=at position .65 with
        {\arrow[rotate=-11,scale=1.5,>=Stealth]{>}}}]  (abr) to[bend right] (abl);
        
    \node (alt) at (-2.1,.5) {};
    \node (alb) at (-2.1, -.5) {};
    \draw[nodecol,dashed, postaction={decorate}, decoration={markings,mark=at position .65 with
        {\arrow[rotate=11,scale=1.5,>=Stealth]{>}}}]  (alt) to[bend left] (alb);
}

\def\dualgraphsetup{
    \node (grtl) at (-1.75,1.75) {};
    \node (grbl) at (-1.75,-1.75) {};
    \node (grtr) at (1.75,1.75) {};
    \node (grbr) at (1.75,-1.75) {};
    
    \draw[nodecol, dashed, postaction={decorate},decoration={
        markings,
        mark=at position .55 with {\arrow[scale=2.5,>=stealth]{>}}}
      ] (grtl) -- (grtr);
    \draw[nodecol, dashed, postaction={decorate},decoration={
        markings,
        mark=at position .55 with {\arrow[scale=2.5,>=stealth]{>}}}
      ]  (grtr) -- (grbr);
    \draw[nodecol, dashed, postaction={decorate},decoration={
        markings,
        mark=at position .55 with {\arrow[scale=2.5,>=stealth]{>}}}
      ]   (grbr) -- (grbl);
    \draw[nodecol, dashed, postaction={decorate},decoration={
        markings,
        mark=at position .45 with {\arrowreversed[scale=2.5,>=stealth]{>}}}
      ]   (grbl) -- (grtl);
    
    \filldraw[compcol] (grtl) circle (3pt);
    \filldraw[compcol] (grtr) circle (3pt);
    \filldraw[compcol] (grbl) circle (3pt);
    \filldraw[compcol] (grbr) circle (3pt);
}

\def\fig1 {
	\def\toplabel{$1$};
	\def\rightlabel{$-2$};
	\def\bottomlabel{$0$};
	\def\leftlabel{$1$};
	
	\begin{scope}
		\fourcurvessetup
		
		\node[nodecol] at (0,3.2) {\toplabel};
        \node[nodecol] at (3.2,0) {\rightlabel};
        \node[nodecol] at (0,-3.2) {\bottomlabel};
        \node[nodecol] at (-3.2,0) {\leftlabel};
		
		\draw[compcol,postaction={decorate}, decoration={markings, 
			mark=at position .4 with \pt, 
			mark=at position .5 with \pt}]  plot[smooth, tension=.8] coordinates {(vtl) (vt) (vr) (vrb)};
		\node[compcol] at (.85,2) {$1$};
		\node[compcol] at (1.4,1.5) {$-4$};
		
		\draw[compcol,postaction={decorate}, decoration={markings, 
			mark=at position .4 with \pt}]  plot[smooth, tension=.8] coordinates {(vrt) (vr) (vb) (vbl)};
		\node[compcol] at (1.8,-1.1) {$3$};
		
		\draw[compcol] plot[smooth, tension=.8] coordinates {(vbr) (vb) (vl) (vlt)};
		
		\draw[compcol,postaction={decorate}, decoration={markings, 
			mark=at position .38 with \pt, 
			mark=at position .6 with \pt}]  
			plot[smooth, tension=.8] coordinates {(vlb) (vl) (vt) (vtr)};
		\node[compcol] at (-2,1) {$2$};
		\node[compcol] at (-1,2) {$-2$};
		\fourcurvesfinish
	\end{scope}
	\begin{scope}[xshift=8cm]
		\dualgraphsetup
		\node[compcol] at (-2,2) {$0$};
		\node[compcol] at (2,2) {$-3$};
		\node[compcol] at (2,-2) {$3$};
		\node[compcol] at (-2,-2) {$0$};
		\node[nodecol] at (0,2.2) {\toplabel};
		\node[nodecol] at (2.2,0) {\rightlabel};
		\node[nodecol] at (0,-2.2) {\bottomlabel};
		\node[nodecol] at (-2.2,0) {\leftlabel};
	\end{scope}
}

\def\eighttriangles { 
	\pgfmathsetmacro{\xspacing}{3.5};
	\pgfmathsetmacro{\yspacing}{-4};
	
	\foreach \x in {0,1,2,3}
	\foreach \y in {0,1}
	{
	    \node (v1\x\y) at (1+\x*\xspacing,1.732+\y*\yspacing) {};
	    \node (v2\x\y) at (0+\x*\xspacing,\y*\yspacing) {};
	    \node (v3\x\y) at (2+\x*\xspacing,\y*\yspacing) {};
	    \node (v1r\x\y) at (1.04+\x*\xspacing,1.732+\y*\yspacing) {};
	    \node (v1l\x\y) at (0.96+\x*\xspacing,1.732+\y*\yspacing) {};
	    \node (v2r\x\y) at (0.04+\x*\xspacing,\y*\yspacing) {};
	    \node (v2l\x\y) at (-0.04+\x*\xspacing,\y*\yspacing) {};
	    \node (v3r\x\y) at (2.04+\x*\xspacing,\y*\yspacing) {};
	    \node (v3l\x\y) at (1.96+\x*\xspacing,\y*\yspacing) {};
	    \node (label\x\y) at (1+\x*\xspacing,-1+\y*\yspacing) {};
	}

	    \begin{scope}[decoration={
	    markings,
	    mark=at position 0.6 with {\arrow{>[length=6pt,width=8pt]}}}]
		    \draw (v1r00) -- (v3r00) node [above=4pt, right=2pt, midway] {$\color{red}1$};
		    \draw[red] (v1l00) -- (v3l00);
		    \draw[red]  (v1r00) -- (v2r00);
		    \draw  (v1l00) -- (v2l00) node [above=4pt, left=2pt, midway] {$\color{red}1$};
		    \draw[red,dashed] (v200) -- (v300)  node [below=4pt, midway] {$1$};
		    \draw[draw=none,postaction={decorate}]  (v100) -- (v300);
		    \draw[draw=none,postaction={decorate}]  (v200) -- (v100);
		    \draw[draw=none, postaction={decorate}] (v200) -- (v300);
		    
		    \draw[red] (v1r10) -- (v3r10) node [above=4pt, right=2pt, midway] {$\color{red}2$};
		    \draw[red] (v1l10) -- (v3l10);
		    \draw[red]  (v1r10) -- (v2r10);
		    \draw[red] (v1l10) -- (v2l10) node [above=4pt, left=2pt, midway] {$\color{red}2$};
		    \draw[red,dashed] (v210) -- (v310)  node [below=4pt, midway] {$1$};
		    \draw[draw=none,postaction={decorate}]  (v110) -- (v310);
		    \draw[draw=none,postaction={decorate}]  (v210) -- (v110);
		    \draw[draw=none, postaction={decorate}] (v210) -- (v310);
		    
		    \draw[red] (v1r20) -- (v3r20) node [above=4pt, right=2pt, midway] {$\color{red}2$};
		    \draw[red] (v1l20) -- (v3l20);
		    \draw[red]  (v1r20) -- (v2r20);
		    \draw  (v1l20) -- (v2l20) node [above=4pt, left=2pt, midway] {$\color{red}1$};
		    \draw[red,dashed] (v220) -- (v320)  node [below=4pt, midway] {$1$};
		    \draw[draw=none,postaction={decorate}]  (v120) -- (v320);
		    \draw[draw=none,postaction={decorate}]  (v220) -- (v120);
		    \draw[draw=none, postaction={decorate}] (v220) -- (v320);
		    
		    \draw (v1r30) -- (v3r30) node [above=4pt, right=2pt, midway] {$\color{red}1$};
		    \draw[red] (v1l30) -- (v3l30);
		    \draw[red]  (v1r30) -- (v2r30);
		    \draw[red]  (v1l30) -- (v2l30) node [above=4pt, left=2pt, midway] {$\color{red}2$};
		    \draw[red,dashed] (v230) -- (v330) node [below=4pt, midway] {$1$};
		    \draw[draw=none,postaction={decorate}]  (v130) -- (v330);
		    \draw[draw=none,postaction={decorate}]  (v230) -- (v130);
		    \draw[draw=none, postaction={decorate}] (v230) -- (v330);
		    
		    \draw[red,dashed] (v1r01) -- (v3r01) node [above=4pt, right=2pt, midway] {$\color{red}2$};
		    \draw[red,dashed] (v1l01) -- (v3l01);
		    \draw[red]  (v1r01) -- (v2r01);
		    \draw[red]  (v1l01) -- (v2l01) node [above=4pt, left=2pt, midway] {$\color{red}2$};
		    \draw[red] (v201) -- (v301)  node [below=4pt, midway] {$1$};
		    \draw[draw=none,postaction={decorate}]  (v301) -- (v101);
		    \draw[draw=none,postaction={decorate}]  (v201) -- (v101);
		    \draw[draw=none, postaction={decorate}] (v201) -- (v301);
		    
		    \draw[red] (v1r11) -- (v3r11) node [above=4pt, right=2pt, midway] {$\color{red}2$};
		    \draw[red] (v1l11) -- (v3l11);
		    \draw[red,dashed]  (v1r11) -- (v2r11);
		    \draw[red,dashed]  (v1l11) -- (v2l11) node [above=4pt, left=2pt, midway] {$\color{red}2$};
		    \draw[red] (v211) -- (v311)  node [below=4pt, midway] {$1$};
		    \draw[draw=none,postaction={decorate}]  (v311) -- (v111);
		    \draw[draw=none,postaction={decorate}]  (v111) -- (v211);
		    \draw[draw=none, postaction={decorate}] (v211) -- (v311);
		    
		    \draw[red,dashed] (v1r21) -- (v3r21) node [above=4pt, right=2pt, midway] {$\color{red}2$};
		    \draw[red,dashed] (v1l21) -- (v3l21);
		    \draw[red]  (v1r21) -- (v2r21);
		    \draw  (v1l21) -- (v2l21) node [above=4pt, left=2pt, midway] {$\color{red}1$};
		    \draw[red] (v221) -- (v321)  node [below=4pt, midway] {$1$};
		    \draw[draw=none,postaction={decorate}]  (v321) -- (v121);
		    \draw[draw=none,postaction={decorate}]  (v221) -- (v121);
		    \draw[draw=none, postaction={decorate}] (v221) -- (v321);
		    
		    \draw (v1r31) -- (v3r31) node [above=4pt, right=2pt, midway] {$\color{red}1$};
		    \draw[red] (v1l31) -- (v3l31);
		    \draw[red,dashed]  (v1r31) -- (v2r31);
		    \draw[red,dashed]  (v1l31) -- (v2l31) node [above=4pt, left=2pt, midway] {$\color{red}2$};
		    \draw[red] (v231) -- (v331)  node [below=4pt, midway] {$1$};
		    \draw[draw=none,postaction={decorate}]  (v331) -- (v131);
		    \draw[draw=none,postaction={decorate}]  (v131) -- (v231);
		    \draw[draw=none, postaction={decorate}] (v231) -- (v331);
	    \end{scope}

	\foreach \x in {0,1,2,3}
	\foreach \y in {0,1}
	{   
	    \node[circle,draw,fill=white, inner sep=6pt] at (v2\x\y) {};
	    \node[circle,draw,fill=white, inner sep=6pt] at (v3\x\y) {};
	    \node[circle,draw,fill=white,double,double distance=1.7pt, inner sep=6pt] at (v1\x\y) {};
    }
    
    \node at (label00) {};
    \node at (label10) {(trivial)};
    \node at (label20) {};
    \node at (label30) {};
    \node at (label01) {(trivial)};
    \node at (label11) {(trivial)};
    \node at (label21) {};
    \node at (label31) {};
    
    \node at (v100) {$0$};
    \node at (v200) {$0$};
    \node at (v300) {$1$};
    
    \node at (v110) {$0$};
    \node at (v210) {$-1$};
    \node at (v310) {$2$};
    
    \node at (v120) {$-1$};
    \node at (v220) {$0$};
    \node at (v320) {$2$};
    
    \node at (v130) {$1$};
    \node at (v230) {$-1$};
    \node at (v330) {$1$};
    
    \node at (v101) {$2$};
    \node at (v201) {$-1$};
    \node at (v301) {$0$};
    
    \node at (v111) {$0$};
    \node at (v211) {$1$};
    \node at (v311) {$0$};
    
    \node at (v121) {$1$};
    \node at (v221) {$0$};
    \node at (v321) {$0$};
    
    \node at (v131) {$-1$};
    \node at (v231) {$1$};
    \node at (v331) {$1$};
    
    \draw [decorate,decoration={brace,amplitude=10pt},xshift=-4pt,yshift=0pt]
	(6.25,-5.5) -- (-0.25,-5.5) node [black,midway,yshift=-0.6cm] 
	{$4$ balanced edge sub-weightings};
}

\title{An arithmetic variant of Raynaud's theorem}
\author{Jonathan Love and Libby Taylor}
\date{\today}

\begin{document}

\maketitle

\begin{abstract}
    It is well known that for a regular semistable curve $\mathfrak X$ over a DVR with algebraically closed residue field, the spanning trees of the dual graph of the special fiber of $\mathfrak X$ are in bijection with components of the special fiber of the N\'eron model of the Jacobian of $\mathfrak X$.  We prove a generalization of this fact that does not require the residue field to be algebraically closed, using a combinatorially enriched version of the dual graph to encode arithmetic information about divisors on $\mathfrak X$.
\end{abstract}

\section{Introduction}

Kirchhoff's matrix-tree theorem states that for any graph $G$, the set of spanning trees of $G$ is in bijection with the elements of $\Pic^0(G)$ (degree $0$ divisors on $G$ modulo principal divisors; see Section~\ref{sec:background} for a full definition).  The original proof~\cite{kirchhoff} is almost entirely linear-algebraic, and does not give a combinatorially defined bijection between these sets.  Bernardi provided such a combinatorial interpretation~\cite{bernardi,bernardi2}, in which he produces a (non-canonical) action of $\Pic^0(G)$ on the set of spanning trees (discussed in Section~\ref{sec:background}).

On the algebrogeometric side, Raynaud's theorem gives a description of the component group of the generalized Jacobian of a curve. For any smooth curve $X$ over the fraction field of a DVR $R$ with algebraically closed residue field, let $\mathfrak X$ denote a regular semistable model of $X$ over $R$. Then $\Phi_X$, the geometric component group of the special fiber of the N\'eron model of the Jacobian of $X$, admits a canonical isomorphism with $\Pic^0(G)$, where $G$ is the dual graph of the special fiber of $\mathfrak X$.  (For a proof of this fact, see Appendix A of~\cite{matt}, in which the relation between the arithmetic and tropical geometry of this problem is spelled out in detail.)

By combining the two theorems above, we get a relationship between divisors of $X$ and spanning trees of the dual graph $G$ of the special fiber of $\mathfrak X$. We will investigate this relationship directly, giving new proofs for some known results and providing analogues of these results in the case that the residue field of the DVR is not algebraically closed. Specifically, using intersection theory on $\mathfrak X$, we define a subgroup $\Pic^{(0)}(X)$ of $\Pic^0(X)$ and a subgroup $\Pic_b^0(G)$ of $\Pic^0(G)$, and prove the following result:
\begin{restatable}{rethm}{arithraynaud}\label{thm:arithraynaud}
    Let $R$ be a henselian DVR with fraction field $K$ and residue field $k$. Let $X$ be a smooth curve over $K$, and $\mathfrak{X}$ a regular semistable model for $X$ over $R$. Then
    \[\Pic^0(X)/\Pic^{(0)}(X)\cong \Pic_b^0(G).\]
\end{restatable}
This can be thought of as an ``arithmetic variant'' of Raynaud's Theorem, which can be stated and proved any reference to N\'eron models or Jacobians. We then use the theory of N\'eron models to establish an isomorphism $\Pic^0(X)/\Pic^{(0)}(X)\cong\Phi_X$ under appropriate conditions, allowing us to prove the following (geometric) variant of Raynaud's Theorem that does not require $k$ to be algebraically closed:

\begin{restatable}{rethm}{raynaud}\label{thm:raynaud}
    Let $R$ be a DVR with fraction field $K$ and residue field $k$. Let $X$ be a smooth, geometrically irreducible curve over $K$, and $\mathfrak{X}$ a regular semistable model for $X$ over $R$. Suppose every irreducible component of $\mathfrak{X}_k$ has a smooth $k$-point. Then $\Pic^0(G) \cong \Phi_X$, where $\Phi_X$ is the component group of the special fiber of the N\'eron model of the Jacobian of $X$.
\end{restatable}

Here ``every irreducible component has a smooth $k$-point'' means smooth as a point in the component, not necessarily as a point in $\mathfrak X_k$ (that is, the $k$-point may be an intersection point with another component).

It should be noted that Theorem~\ref{thm:raynaud} is not new; it was originally proved in~\cite{boschliu}, and we have recreated their results in somewhat different language.  We give a new application of this theorem, though, and use it to define a set of combinatorial objects which are a torsor under the arithmetic component group.

In general, the relationship between line bundles on the curve and divisors on $G$ is not as straightforward as it is when $k$ is algebraically closed. In particular, arithmetic information about $X$ will translate into ``weights'' that are attached to the vertices and edges of $G$. We study these weighted graphs in Section~\ref{sec:weightedgraphs}, and prove a weighted analogue of the matrix-tree theorem (Theorem~\ref{thm:countgraphcosets}), allowing us to establish a torsor action of $\Pic^0(X)/\Pic^{(0)}(X)$ on a set of combinatorially enriched spanning trees (Corollary~\ref{cor:spantreedecomp}).




The paper is organized as follows. In Section~\ref{sec:background}, we introduce the required background in graph theory. In Section~\ref{sec:weightedgraphs}, we introduce weighted graphs and a number of related notions, and prove some results about them, including the matrix-tree analogue Theorem~\ref{thm:countgraphcosets}. In Section~\ref{sec:specialization} we begin discussing semistable curves over a DVR, focusing on properties of the specialization map, a map from divisors on the curve to divisors on the dual graph $G$ of the special fiber. We apply these properties in Section~\ref{sec:arithcomponent} to prove Theorems~\ref{thm:arithraynaud} and~\ref{thm:raynaud}, and combine these results with the results about weighted graphs to prove Corollary~\ref{cor:spantreedecomp}.

\subsection{Acknowledgements}

The authors would like to thank Qing Liu for helpful comments on an earlier version of this paper, and for bringing~\cite{boschliu} to our attention.  The second author would like to thank Trader Joe's for having provided the coffee which formed the raw material for much of this paper.   The first named author would, accordingly, like to thank his mother, who seems to consistently forget that she already sent him a bag of coffee beans relatively recently.

\section{Background on chip-firing}\label{sec:background}

In this section, we will review the necessary graph theory; these theorems will later be applied to the dual graph of the special fiber of a regular proper model of a curve.

Let $G$ be a directed graph on $n$ vertices.  Its $n \times n$ signed adjacency matrix $A$ is defined to to have its $(i,j)$-th entry in $A$ equal to $1$ if there is a directed edge from $v_i$ to $v_j$; $-1$ if there is a directed edge from $v_j$ to $v_i$; and $0$ if no edge exists between the two vertices.  If there are multiple edges between $v_i$ and $v_j$, then the matrix entry is equal to the number of edges oriented $v_i$ to $v_j$ minus the number of edges oriented $v_j$ to $v_i$.

\smallskip

A divisor on a graph $G=(V,E)$ is a function $D: V\to \Z$, which we write in the form $\sum_{v\in V} a_v[v]$ for $a_v\in \Z$.  The set of all divisors of a graph $G$ is denoted $\Div(G)$. The degree of the divisor $D$ is defined as $\deg(D)=\sum_{v\in V} a_v$.  The set of all divisors of degree $k$ is denoted $\Div^k(G)$.  There is a (non-canonical) bijection $\Div^k(G)\to \Div^0(G)$ defined by $D\mapsto D-D_0$ for some fixed reference divisor $D_0\in \Div^k(G)$.  

\smallskip

The \emph{Laplacian operator} on a graph $G$, denoted $\Delta :\Div(G)\to \Div(G)$, is defined by 
\[\Delta (f):=\sum_{v\in V}\left(\sum_{\substack{e\in E \\ e=\{v,w\} }}(f(v)-f(w))\right)[v].\]
If $f$ is $1$ at a single vertex $v$ and $0$ everywhere else, $\Delta(f)$ is called a ``chip-firing move at $v$.'' The group $\text{Prin}(G)$ of principal divisors is the image of the Laplacian operator; note that it is generated by chip-firing moves. We have $\text{Prin}(G)\subseteq \Div^0(G)$, and both $\text{Prin}(G)$ and $\Div^0(G)$ are free abelian groups of rank $n-1$,  so
\[
\Pic^0(G)=\Div^0(G)/\text{Prin}(G)
\]
is a finite group called the \textit{Jacobian} of $G$.

The Jacobian of $G$ is in bijection with the set $T(G)$ of spanning trees of $G$; this is (one form of) the classical matrix-tree theorem. While the theorem was originally proved in~\cite{kirchhoff} using linear algebra (hence the name), it was proved much more recently that there is a combinatorially-defined bijection between these two sets. This result, due to Bernardi~\cite{bernardi2}, follows by establishing a bijection between each of these sets and a particular set of divisors on $G$. We summarize this construction below.

We assign to $G$ a \emph{ribbon structure}: that is, for each vertex $v$, we order the edges incident to $v$ up to cyclic permutation. Intuitively, we may consider drawing $G$ in the plane (with edges allowed to cross each other), and define a ribbon structure by the counterclockwise ordering of edges around each vertex. Now fix a vertex $q\in G$ and an edge $e_0$ incident to $q$. An orientation of $G$ is \emph{$q$-connected} if for every vertex $v\in G$, there is a directed path from $q$ to $v$. 

Given a spanning tree $T$ of $G$, we can define a $q$-connected orientation of $T$ using the following procedure (called a ``tour'' of the spanning tree), beginning with $(v,e):=(q,e_0)$:
\begin{itemize}
    \item If $e$ is not in $T$:
    \begin{enumerate}
        \item If $e$ has not been assigned an orientation yet, orient $e$ towards $v$.
        \item Re-assign $e$ to be the next edge in the ordering around $v$. 
    \end{enumerate}
    \item Otherwise (if $e$ is in $T$):
    \begin{enumerate}
        \item If $e$ has not been assigned an orientation yet, orient $e$ away from $v$.
        \item Re-assign $v$ to be the other vertex incident to $e$.
    \end{enumerate}
\end{itemize}
See Figure~\ref{fig:spanningtrees} for examples of $q$-connected orientations arising from tours of spanning trees. Further discussion of tours (including a visualization) can be found in~\cite{bernardi}. Tours of spanning trees eventually orient every edge of $G$~\cite[Lemma 5]{bernardi}, and the resulting orientation is $q$-connected~\cite[Lemma 5]{bernardi2}.

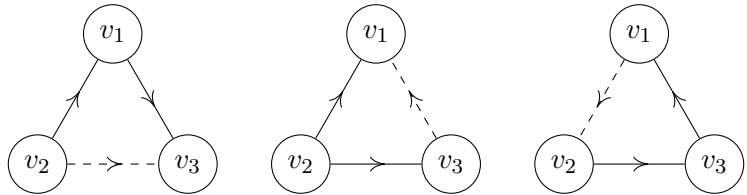
\begin{figure}[h]
    \centering
    \begin{tikzpicture}
    	\begin{scope}[decoration={
        markings,
        mark=at position 0.6 with {\arrow{>[length=2mm,width=2mm]}}}]
        \node[circle,draw] (v2) at (0,0) {$v_2$};
        \node[circle,draw] (v3) at (2,0) {$v_3$};
        \node[circle,draw] (v1) at (1,1.732) {$v_1$};
        \draw[dashed,postaction={decorate}] (v2) to (v3);
        \draw[postaction={decorate}]  (v1) -- (v3);
        \draw[postaction={decorate}]  (v2) -- (v1);
        
        \node[circle,draw] (v21) at (3.5,0) {$v_2$};
        \node[circle,draw] (v31) at (5.5,0) {$v_3$};
        \node[circle,draw] (v11) at (4.5,1.732) {$v_1$};
        \draw[postaction={decorate}]  (v21) to (v31);
        \draw[dashed,postaction={decorate}]  (v31) -- (v11);
        \draw[postaction={decorate}]  (v21) -- (v11);
        
        \node[circle,draw] (v22) at (7,0) {$v_2$};
        \node[circle,draw] (v32) at (9,0) {$v_3$};
        \node[circle,draw] (v12) at (8,1.732) {$v_1$};
        \draw[postaction={decorate}]  (v22) -- (v32);
        \draw[postaction={decorate}]  (v32) -- (v12);
        \draw[dashed,postaction={decorate}]  (v12) -- (v22);
        \end{scope}
    \end{tikzpicture}
    \caption{Spanning Trees for the Triangle Graph, and their corresponding $v_2$-connected orientations; the starting edge $e_0$ in each case is $\{v_2,v_1\}$.}
    \label{fig:spanningtrees}
\end{figure}
    
Now to any orientation $O$, we can associate a divisor $D_O\in \Div^{g-1}(G)$ by placing a coefficient of $\text{indeg}(v)-1$ at each vertex $v$, where $g$ denotes the combinatorial genus of $G$ and $\text{indeg}(v)$ denotes the number of edges pointing towards $v$ according to $O$. A divisor $D$ is \emph{$q$-orientable} if it equals $D_O$ for a $q$-connected orientation $O$. Bernardi proves the following result:

\begin{prop}[{\cite[Theorem 46(5)]{bernardi2}}]\label{prop:q-conn}

The map from spanning trees to $q$-orientable divisors described above is a bijection.\footnote{Bernardi actually establishes a bijection between spanning trees and ``$q$-connected outdegree sequences,'' that is, divisors of the form $\sum_{v\in G} \text{outdeg}(v)[v]$ where the out-degree is defined by a $q$-connected orientation $O$. There is a natural bijection from $q$-connected outdegree sequences to $q$-orientable divisors given by $D\mapsto \sum_{v\in G}(\deg(v)-1)[v]-D$.}

\end{prop}

\begin{rem}
    The set of $q$-connected orientations may not be in bijection with either of the sets in Proposition~\ref{prop:q-conn}. In the situation of Figure~\ref{fig:spanningtrees}, the clockwise orientation of edges is a $q$-connected orientation which does not come from a tour of any spanning tree (the map from spanning trees to $q$-connected orientations is not surjective), and it defines the same divisor as the counterclockwise orientation (the map from $q$-connected orientations to $q$-orientable divisors is not injective).
\end{rem}

Combining Proposition~\ref{prop:q-conn} with the following result, we derive the desired bijection between the set of spanning trees and $\Pic^0(G)$.

\begin{thm}[{\cite[Theorem 4.13]{ABKS}}]\label{thm:abks}
    If $D$ is a degree $g-1$ divisor on $G$, then there is a unique $q$-orientable divisor that is equivalent to $D$ (i.e. differs from $D$ by a sum of chip-firing moves).
\end{thm}

The map from $\Div^{g-1}(G)$ to $\Pic^0(G)$ is defined by subtracting a reference divisor $D_0\in \Div^{g-1}(G)$ from each $D\in \Div^{g-1}(G)$.  The reference divisor is taken to be the $q$-orientable divisor associated to some spanning tree $T_0$ of $G$, which allows $T_0$ to be considered as an ``identity element'' of $T(G)$ in this bijection.  This produces a bijection between $\Div^{g-1}(G)/\text{Prin}(G)$ and $\Div^0(G)/\text{Prin}(G) \cong \Pic^0(G)$.  

Thus the Bernardi map gives a combinatorially defined bijection between spanning trees of $G$ and elements of $\Pic^0(G)$, which 
factors as $T(G) \to \Div^{g-1}(G)/\Prin(G) \to \Pic^0(G)$.  Note that this bijection is not canonical, as it depends on the choice of $T_0$ and the choice of $q$.

\smallskip


If $R$ is a DVR with algebraically closed residue field, and $\mathfrak X$ is a regular semistable curve over $R$, then we can study properties of $\mathfrak X$ by taking the dual graph $G$ of the special fiber. In particular, there is a \emph{specialization map} (Section~\ref{sec:specialization}) from divisors on $\mathfrak X$ to divisors on $G$, so the graph theory described above can be used to study the geometry of $\mathfrak X$.

If the residue field is not algebraically closed, however, there may not be as nice of a relationship between divisors on $\mathfrak X$ and on $G$; for instance, there may be no divisor on $\mathfrak X$ that has degree $1$ when restricted to a particular component of the special fiber. In the following section, we develop a theory of weighted graphs that can be used to resolve some of these issues.

\section{A Theory of Weighted Graphs}\label{sec:weightedgraphs}

\subsection{Weights and Balanced Divisors}\label{sec:balanceddivs}

Let $G=(V,E)$ be a multigraph. A function $\omega:V\sqcup E\to \Z_{>0}$ will be called a ``weighting'' of $G$, and the pair $(G,\omega)$ is a \emph{weighted graph}. We say that a weighted graph is \emph{pleasant} if $\omega(v)\mid \omega(e)$ whenever $v$ is a vertex on the edge $e$ (see Figure~\ref{fig:weightedgraph}).

\begin{figure}[h]
    \centering
    \begin{tikzpicture}
        \node[circle,draw] (v2) at (0,0) {$v_2$};
        \node[circle,draw] (v3) at (2,0) {$v_3$};
        \node (v1) at (1,1.732) {};
        \draw[double, double distance=2pt]  (v2) to [bend right] (v3);
        \draw  (v2) to [bend left] (v3);
        \draw[double, double distance=2pt]  (v3) -- (v1);
        \draw[double, double distance=2pt]  (v2) -- (v1);
        \node[circle,fill=white, double,draw,double distance=2pt] (v1) at (1,1.732) {$v_1$};
        
        \node[circle,draw] (w2) at (5,0) {$w_2$};
        \node[circle,draw] (w3) at (7,0) {};
        \node (w1) at (6,1.732) {};
        \draw[double, double distance=2pt]  (w2) to [bend right] (w3);
        \draw  (w2) to [bend left] (w3);
        \draw[double, double distance=2pt]  (w3) -- (w1);
        \draw[double, double distance=2pt]  (w2) -- (w1);
        \node[circle,fill=white,double,draw,double distance=2pt] (w1) at (6,1.732) {$w_1$};
        \node[circle,fill=white,double,draw,double distance=2pt] (w3) at (7,0) {$w_3$};
    \end{tikzpicture}
    \caption{A multigraph with two distinct weightings. Weights are represented as the number of lines along a vertex or edge. The weighting on the left is pleasant, while the weighting on the right is not.}
    \label{fig:weightedgraph}
\end{figure}
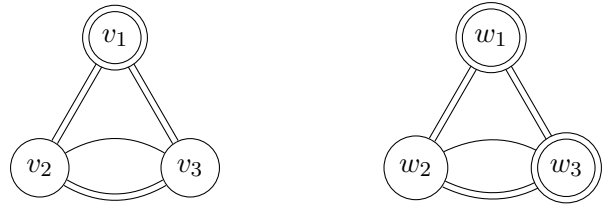

A divisor on on a weighted graph $G$ is a function $D:V\to \Z$; the set of all divisors is denoted $\Div(G)$. We say a divisor is \emph{balanced} if $\omega(v)\mid D(v)$ for all $v$; the group of all balanced divisors is denoted $\Div_b(G)$. We define the set of \emph{unbalancings} of $G$ to be the quotient:
\[\Div_u(G):=\Div(G)/\Div_b(G)\cong\prod_{v\in V} \Z/\omega(v)\Z.\]
Using the examples from Figure~\ref{fig:weightedgraph}, a divisor $D$ on the weighted graph on the left is balanced if and only if the $D(v_1)$ is even, and the set of unbalancings is $\Z/2\Z$.

Let $\omega(G):=\gcd\{\omega(v):v\in V\}$. Taking degrees of divisors, we obtain the following commutative and exact diagram:
\[\begin{tikzcd}
    & 0 \arrow[d]  & 0 \arrow[d]  & 0 \arrow[d] \\
    0 \arrow[r] & \Div_b^0(G) \arrow[r] \arrow[d]  & \Div^0(G) \arrow[r] \arrow[d]  & \Div_u^0(G) \arrow[r] \arrow[d]  & 0 \\
    0 \arrow[r] & \Div_b(G) \arrow[r] \arrow[d, "\deg"]  & \Div(G) \arrow[r] \arrow[d, "\deg"]  & \Div_u(G) \arrow[r] \arrow[d]  & 0 \\
    0 \arrow[r] & \omega(G)\Z \arrow[r] \arrow[d] & \Z \arrow[r] \arrow[d] & \Z/\omega(G)\Z \arrow[r] \arrow[d] & 0 \\
                & 0                              & 0                            & 0.                             &  
    \end{tikzcd}
\]

Now given a weighted graph and some function $f:V\to\Z$, we define the (weighted) Laplacian
\[\Delta (f):=\sum_{v\in V}\left(\sum_{\substack{e\in E \\ e=\{v,w\} }}\omega(e)(f(v)-f(w))\right)[v]\in\Div^0(G).\]
For example, using the pleasantly weighted graph from Figure~\ref{fig:weightedgraph}, the Laplacian of the indicator function of $v_3$ is the divisor $D$ with $D(v_1)=-2$, $D(v_2)=-3$, and $D(v_3)=5$. Let $\Prin(G)$ denote the image of $\Delta$; elements of $\Prin(G)$ are called \emph{principal divisors} or \emph{chip-firing moves} on $G$, and two divisors that differ by a principal divisor are \emph{chip-firing equivalent}.

Note that the definition of ``balanced divisor'' depends only on the weights of vertices, and the definition of ``principal divisor'' depends only on the weights of edges. Pleasant weightings allow us to relate these two notions: if $G$ is pleasantly weighted, then all principal divisors are balanced.

\subsection{Spanning Trees and Sub-weightings}

Given a spanning tree $T=(V,E_T)$ of $G$, we define an \emph{edge sub-weighting} of $T$ to be a function $\sigma:E\to\Z$ with the property that $\sigma(e)=\omega(e)$ if $e\notin E_T$, and $1\leq \sigma(e)\leq\omega(e)$ for all $e\in E_T$. The \emph{trivial sub-weighting} is the one with $\sigma(e)=\omega(e)$ for all $e\in E$. 

Assume for the sake of simplicity that $G$ is connected; we explain how to remove this hypothesis at the end of this section. Fix a vertex $q\in V$ and a ribbon structure of $G$. Given a spanning tree $T$ of $G$, we can construct a $q$-connected orientation of $G$ as described in Section~\ref{sec:background}. If we are additionally given an edge sub-weighting for $T$, we can then define a divisor 
\begin{align*}
    D_{T,\sigma}&:=\sum_{e:w\to v\in E}\Big(\sigma(e)[v]+(\omega(e)-\sigma(e))[w]\Big)-\sum_{v\in V}\omega(v)[v]\\
    &=\sum_{v\in V}\left(\sum_{e\text{ to }v}\sigma(e)+\sum_{e\text{ from }v}(\omega(e)-\sigma(e))-\omega(v)\right)[v].
\end{align*}
See Figure~\ref{fig:eighttrees} for an illustration of this definition.
Note that $D_{T,\sigma}\in\Div^{g-1}(G)$, where $g:=\left(\sum_{e} \omega(e)\right)-\left(\sum_{v} \omega(v)\right)+1$ is the \emph{weighted genus} of $G$. Further, we observe that $D_{T,\sigma}$ is balanced if and only if for all $v\in V$, we have
\[\left.\left(\sum_{e\text{ to }v} \sigma(e)-\sum_{e\text{ from }v}\sigma(e)\right)\right| \omega(v).\]
If this condition is satisfied, we call $\sigma$ a \emph{balanced edge sub-weighting} for $T$. For any spanning tree, the trivial sub-weighting is always balanced, but there may be nontrivial balanced edge sub-weightings.
\begin{ex}
    The triangle graph has three spanning trees, each defining a $v_2$-connected orientation as in Figure~\ref{fig:spanningtrees}. We assign a pleasant weighting to this graph: let vertex $v_1$ and edges $\{v_1,v_2\}$ and $\{v_1,v_3\}$ have weight $2$, and all other edges and vertices have weight $1$. 
    With this weighting, one of the spanning trees has $4$ sub-weightings, and the other two spanning trees have $2$ sub-weightings each, for a total of $8$ possible edge sub-weightings. These are listed in Figure~\ref{fig:eighttrees}. 
    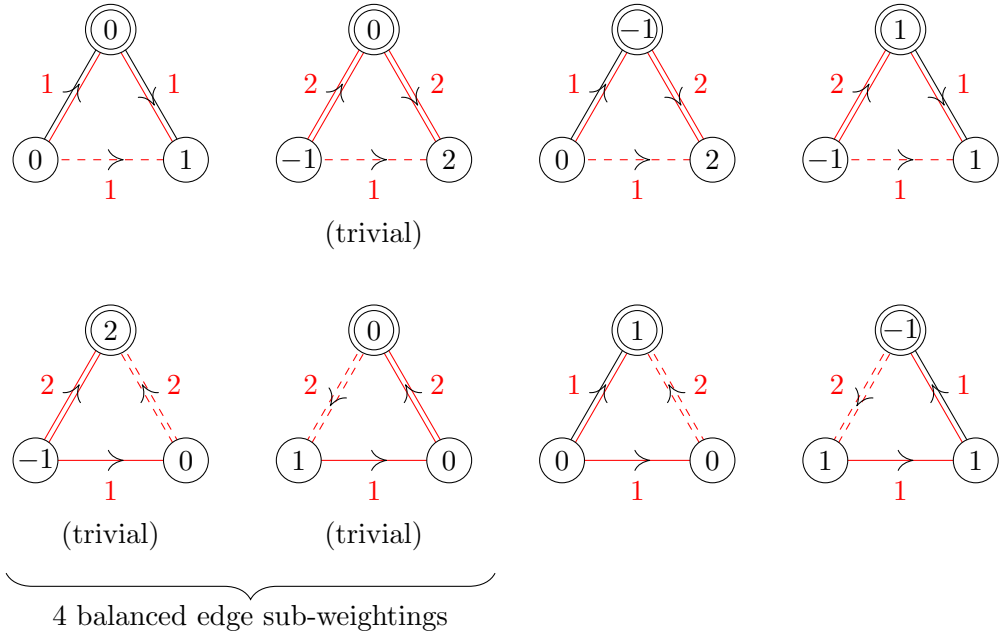
\begin{figure}[h]
        \centering
        \begin{tikzpicture}
            \eighttriangles
        \end{tikzpicture}
        \caption{Edge Sub-weightings for a given pleasant weighting. Edge sub-weightings are labelled with red. Vertices are labeled with the values of the divisor $D_{T,\sigma}$.}
        \label{fig:eighttrees}
    \end{figure}
    For each of the sub-weightings $\sigma$, we label the edges in red with $\sigma(e)$, and the vertices with $D_{T,\sigma}(v)$. Under this labeling, the value $D_{T,\sigma}(v)$ can be computed as ``red lines into $v$ plus black lines out of $v$ minus the weight of $v$.'' Notice that the four divisors on the left are balanced (the coefficient of $v_1$ is even); in particular, we see in the very top left an example of a nontrivial balanced sub-weighting (the other three balanced sub-weightings are all trivial).
    
    By the following theorem, every divisor of degree $1$ on $G$ is chip-firing equivalent to exactly one of these $8$ divisors, and every balanced divisor of degree $1$ is chip-firing equivalent to exactly one of the four on the left.
\end{ex}

\begin{thm}\label{thm:countgraphcosets}
    Let $(G,\omega)$ be a pleasantly weighted graph. Then
    \[\{D_{T,\sigma}:\text{$T$ a spanning tree for $G$, $\sigma$ an edge sub-weighting for $T$}\}\]
    forms a complete set of coset representatives for $\Prin(G)$ in $\Div^{g-1}(G)$, and 
    \[\{D_{T,\sigma}:\text{$T$ a spanning tree for $G$, $\sigma$ a balanced edge sub-weighting for $T$}\}\]
    forms a complete set of coset representatives for $\Prin(G)$ in $\Div_b^{g-1}(G)$.
\end{thm}

\begin{rem}
    Theorem~\ref{thm:countgraphcosets} is a direct generalization of the classical matrix-tree theorem that the set of spanning trees of $G$ is in bijection with elements of $\Pic^0(G)$. The classical theorem arises when we take $\omega$ to be a trivial weighting (i.e.\ every vertex and edge has weight $1$).
\end{rem}

\begin{proof}
    Let $\widehat{G}=(V,\widehat{E})$ be the unweighted graph obtained by replacing each edge $e$ of $G$ by a collection of $\omega(e)$ parallel edges $e_1,\ldots, e_{\omega(e)}$. Under the natural identification $\Div(G)\cong\Div(\widehat{G})$, the principal divisors are also identified. 
    
    To each spanning tree $\widehat{T}$ of $\widehat{G}$, we will define a spanning tree $T$ of $G$ together with an edge sub-weighting of $T$. For each edge $e$ of $G$, which has been replaced with the parallel edges $e_1,\ldots, e_{\omega(e)}$ in $\widehat{G}$, note that $\widehat{T}$ can include at most one of these edges; include $e$ in $T$ if and only if some $e_i$ is in $\widehat{T}$. Now consider the $q$-connected orientation of $\widehat{G}$ defined by $\widehat{T}$. If $e\in T$, let $\sigma(e)$ be the number of edges $e_1,\ldots, e_{\omega(e)}$ pointing in the same direction as $e_i$ in the $q$-connected orientation resulting from $\widehat{T}$; this means that there will be $\omega(e)-\sigma(e)$ edges pointing in the opposite direction. 
    
    Now set 
    \[D_{\widehat{T}}:= \sum_{v\in V}(\text{indeg}(v)-1)[v]=\sum_{e:w\to v\in \widehat{E}}[v]-\sum_{v\in V}[v].\]
    By Theorem~\ref{thm:abks}, the divisors $D_{\widehat{T}}$ give a complete set of coset representatives for $\Prin(\widehat{G})$ in $\Div^{\widehat{g}-1}(\widehat{G})$, where $\widehat{g}$ is the unweighted first betti number of $\widehat{G}$. But by construction of $T$ and $\sigma$ from $\widehat{T}$, we can see that
    \[D_{T,\sigma}=D_{\widehat{T}}-\sum_{v\in V}(\omega(v)-1)[v].\]
    Since $\sum_{v\in V}(\omega(v)-1)[v]$ is constant for all trees and sub-weightings, this gives us a complete set of $\Prin(G)$--coset representatives in $\Div^{g-1}(G)$.
    
    Finally, since every element of $\Prin(G)$ is balanced, we can obtain $\Prin(G)$--coset representatives in $\Div_b^{g-1}(G)$ simply by taking those edge sub-weighted spanning trees for which $D_{T,\sigma}$ is a balanced divisor.
\end{proof}

\begin{cor}\label{cor:countgraphcosets}
    The number of elements of $\Div^0(G)/\Prin(G)$ is
    \[\sum_{\substack{T \text{ spanning}\\ \text{tree of }G}}\  \prod_{e\in E_T} \omega(e),\]
    and the number of elements of $\Div_b^0(G)/\Prin(G)$ is
    \[\frac{\gcd_{v\in V}\omega(v)}{\prod_{v\in V}\omega(v)}\left(\sum_{\substack{T \text{ spanning}\\ \text{tree of }G}} \prod_{e\in E_T} \omega(e)\right).\]
\end{cor}
\begin{proof}
    Note that $\Div^{g-1}(G)$ is a $\Div^0(G)$--torsor, and so the number of elements of $\Div^0(G)/\Prin(G)$ is equal to the number of $\Prin(G)$--cosets contained in $\Div^{g-1}(G)$. By Theorem~\ref{thm:countgraphcosets}, this is equal to the number of edge sub-weighted spanning trees. Given a spanning tree, we can define an edge sub-weighting on each edge independently, giving us the desired count.
    
    Now observe that 
    \[\frac{\Div^0(G)/\Prin(G)}{\Div_b^0(G)/\Prin(G)}\cong \frac{\Div^0(G)}{\Div_b^0(G)}=\Div_u^0(G),\]
    where $\Div_u^0(G)$ is the kernel of the surjective map $\Div_u(G)\to \Z/\omega(G)\Z$ and hence has $\frac{1}{\omega(G)}\prod_{v\in V}\omega(v)$ elements. This implies the desired conclusion. 
\end{proof}

\begin{rem}\label{rem:disconnected}
    We can generalize these results to the case that $G$ is disconnected. Given a graph $G$, a \emph{maximal spanning forest} of $G$ is a subgraph $F$ with the same vertex set as $G$, such that for all edges of $G$ not in $F$, adding the edge to $F$ will result in a subgraph with a loop. Equivalently, $F$ is the union of spanning trees on each component of $G$. Instead of $q$-connected orientations, we will have $\{q_1,\ldots,q_r\}$-connected orientations, where $\{q_1,\ldots,q_r\}$ contains one vertex from each component of $G$. We can define edge sub-weightings $\sigma$ for a maximal spanning forest $F$ in the same way as for spanning trees, and divisors $D_{F,\sigma}$ can also be defined analogously.
    
    To generalize Theorem \ref{thm:countgraphcosets} and Corollary~\ref{cor:countgraphcosets}, it suffices to replace spanning trees with maximal spanning forests. The groups $\Div^{g-1}(G)$, $\Div^{g-1}_b(G)$, and $\Prin(G)$ each decompose naturally as a direct sum of subgroups on each component of $G$, so choosing a coset representative for $\Prin(G)$ in $\Div^{g-1}(G)$ (resp.\ $\Div^{g-1}_b(G)$) is equivalent to choosing a coset representative for $\Prin(G_i)$ in $\Div^{g-1}(G_i)$ (resp.\ $\Div^{g-1}_b(G_i)$) for each component $G_i$ of $G$. Under this decomposition, each divisor $D_{F,\sigma}$ can be written as a sum of elements $D_{T_i,\sigma_i}$ by restriction to each component $G_i$.
\end{rem}

\section{Intersection Theory and the Specialization Map}\label{sec:specialization}

For the rest of the paper, we assume the following setup. Let $R$ be a discrete valuation ring with field of fractions $K$ and residue field $k$. Let $X$ be a smooth curve over $K$, and $\mathfrak{X}$ a regular semistable model for $X$ over $R$.  By semistable, we mean that the special fiber is reduced with only ordinary double point singularities.\footnote{It is more common to define a semistable curve as one whose special fiber is geometrically reduced.  However, one may blow up any such curve to get one whose special fiber is a normal crossing divisor.  Since blowups do not affect the combinatorics of the special fiber, as per the next lemma, this distinction will be irrelevant for us.} Let $G=(V,E)$ the dual graph of the special fiber $\mathfrak{X}_k$, so that $V=\{C_1,\ldots,C_\gamma\}$ is the set of components of $\mathfrak{X}_k$, and $E$ is the set of nodes.

We will be working with $\Pic^0(G)$; we claim that this group is independent of the choice of semistable regular proper model.  This fact is well-known, but the authors were unable to find a reference, so a proof is given here.

\begin{lem}
    Let $G$ be the dual graph of $\X_k$, the special fiber of the minimal regular proper model, and let $G'$ be the dual graph of $\X'_k$ for $\X'$ any semistable proper regular model.  Then $\Pic^0(G)\cong \Pic^0(G')$.
\end{lem}

\begin{proof}
It suffices to show that $\Pic^0(G)$ is unchanged upon blowing up a point on the special fiber to obtain another semistable model.  There are two cases to consider, corresponding to whether the point $p\in \X_k$ to be blown up is a smooth or nodal point. If $p$ is a smooth point on some component $C_i$ and $P$ denotes the exceptional divisor that comes from blowing up $p$, then $P\cdot C_i=1$ and $P\cdot C_k=0$ for all $k\neq i$. This means that $P$ adds a leaf to the dual graph, which does not change $\Pic^0(G)$.

On the other hand, suppose $p$ is the intersection of two components $C_i$ and $C_j$. Then we have $P\cdot \X_k=0$, $P\cdot C_i=1$ and $P\cdot C_j=1$, so that $P^2=-2$. But $P$ is a copy of $\P^1$, therefore it must be nonreduced. But this means that the model is no longer semistable.
\end{proof}

Let $\Div(X)$ and $\Div(\mathfrak{X})$ denote the set of Cartier divisors on $X$ and $\mathfrak{X}$ respectively (since $X$ and $\mathfrak{X}$ are both regular, we can equivalently consider Weil divisors). Given $D\in \Div(X)$, we can define $\overline{D}\in\Div(\mathfrak{X})$, the ``Zariski closure'' of $D$, by defining it first on prime divisors in the natural way, and then extending linearly. Since $D$ can be recovered from $\overline{D}$ by restriction to the generic fiber, this allows us to identify $\Div(X)$ with a subgroup of $\Div(\mathfrak{X})$, the group of ``horizontal divisors'' on $\mathfrak{X}$. Every divisor on $\mathfrak X$ can be written uniquely as a sum of a horizontal divisor and a ``vertical divisor,'' that is, a divisor supported on $\mathfrak X_k$.

Define the \emph{specialization map} $\rho:\Div(\mathfrak{X})\to\Div(G)$ by
\[\rho(\mathcal{D}):=\sum_{C_i\in V} (C_i\cdot \mathcal{D})[C_i],\]
where $C_i\cdot \mathcal{D}=\deg(\O_{\mathfrak{X}}(\mathcal{D})|_{C_i})$ is the intersection pairing. The map $\rho$ is evidently a homomorphism, and it preserves degrees because the degree of the line bundle $\O_{\mathfrak{X}}(\mathcal{D})$ on $\mathfrak{X}$ is preserved by base change to the special fiber.  Since the intersection pairing is invariant under linear equivalence, $\Prin(\mathfrak X)$ is contained in $\ker\rho$.

We will prove that $\rho$ maps $\Prin(X)$ surjectively onto $\Prin(G)$ (Lemma~\ref{lem:specialprin}). On the other hand, $\Div(X)$ does not necessarily map surjectively onto $\Div(G)$; we will compute the image of $\Div(X)$ in $\Div(G)$ (Lemma~\ref{lem:specialdiv}). Both results require the following lemma. 

\begin{lem}\label{lem:shafarevich}
    Let $S$ be the spectrum of a field or of a DVR. For any divisor on a regular proper curve over $S$, and any finite set of closed points $\{p_1,\dots,p_m\}$ on the curve, there exists a linearly equivalent divisor with support disjoint from $\{p_1,\dots,p_m\}$.  
\end{lem}

\begin{proof}
    A proper curve over a field is projective over the field~\cite[0A26, 0B45]{stacks}, and a regular proper curve over a DVR is projective over the DVR~\cite[0C5P]{stacks}. The desired result is proved in~\cite[Theorem 3.1]{shafar} for smooth projective varieties over a field, but the proof holds for regular projective schemes over a field or a DVR as well.
\end{proof}

\begin{lem}\label{lem:specialprin}
    The specialization map $\rho$ induces a surjection $\Prin(X)\to\Prin(G)$.
\end{lem}

\begin{rem}
    Baker provides a proof for this fact in the case that $k$ is algebraically closed~\cite{matt}. The fact that $\Prin(X)$ maps into $\Prin(G)$ is his Lemma 2.1; we reproduce the argument here, essentially unchanged. The proof he provides for surjectivity (Corollary A.9), however, follows from the surjectivity of $\Div(X)\to\Div(G)$, which will not hold for general $k$. Baker's proof is also a consequence of more powerful results about the Jacobian variety. We provide a more elementary proof that does not depend on $k$ being algebraically closed.
\end{rem}

\begin{proof}
    Suppose we are given $D\in\Prin(X)$, so $D=\text{div}(f)$ for some rational function $f$ on $X$. The open inclusion $X\to\mathfrak{X}$ induces an isomorphism of function fields, so $f$ extends uniquely to a rational function $\widehat{f}$ on $\mathfrak{X}$. Then $\overline{D}-\text{div}(\widehat{f})$ will be a vertical divisor. Now for each irreducible component $C_j$ of $\mathfrak X_k$, $\rho([C_j])$ is the divisor corresponding to a chip-firing move at the vertex $C_j$ (this follows from the fact that $(C_j\cdot \mathfrak X_k)=0$). This shows that $\rho$ maps the set of vertical divisors surjectively onto $\Prin(G)$, so 
    \[\rho(D)=\rho(D-\text{div}(\widehat{f}))\in\Prin(G).\]

    Now suppose we are given an element in $\Prin(G)$, which by the preceding discussion must equal $\rho(\mathcal{V})$ for some vertical divisor $\mathcal{V}$. Let $x_1,\dots,x_m$ denote closed points on the components $C_1,\dots,C_m$ of $\mathfrak X_k$ respectively. Applying Proposition~\ref{lem:shafarevich} to the divisor $\mathcal V$ and the points $x_1,\ldots,x_m$, we can find a divisor $D=\mathcal V+\text{div}(f)$ for some rational function $f$ on $\mathfrak{X}$, such that $D$ does not contain any irreducible component of $\mathfrak{X}_k$ in its support. Then $D$ is a horizontal divisor, and we have $\rho(D)=\rho(\mathcal V)$ because $\rho$ annihilates $\Prin(\mathfrak X)$.
\end{proof}

We now describe how to measure the failure of surjectivity of $\Div(X)\to\Div(G)$. Following~\cite[Section 1.5]{hasstschink}, we define the \emph{index} of a smooth curve $C/k$, $\ind(C)$, to be the greatest common divisor of the degrees of field extensions $k'/k$ such that $C(k')\neq\emptyset$. The degree of any divisor on $C$ will be a multiple of the index. If $C$ is not smooth, we define $\ind(C)$ to be the index of the normalization $\widehat{C}$ of $C$.

\begin{lem}\label{lem:specialdiv}
    Assume $R$ is henselian. The specialization map $\rho$ induces a surjection $\Div(X)\to\Div_b(G)$, where $\Div_b(G)\leq \Div(G)$ consists of the divisors on $G$ such that the coefficient of $[C_i]$ is a multiple of $\ind(C_i)$.
\end{lem}
\begin{rem}
    If we define a weighting on $G$ such that $\omega(C_i):=\ind(C_i)$, then $\Div_b(G)$ is the set of balanced divisors on $G$, as defined in Section~\ref{sec:weightedgraphs}.
\end{rem}
\begin{proof}
    Given $D\in\Div(X)$ and any component $C_i$ of $\mathfrak{X}_k$, observe that $\O_{\mathfrak{X}}(\overline{D})|_{C_i}$ will be a line bundle on $C_i$, and so its degree, $C_i\cdot \overline{D}$, will be a multiple of $\ind(C_i)$. This proves that $\Div(X)$ maps into $\Div_b(G)$.
    
    To show surjectivity, it suffices to show that for any $i$, $\ind(C_i)[C_i]$ is in the image of $\Div(X)$. Let $\widehat{C_i}$ denote the normalization of $C_i$. By definition of index, we can find closed points $\widehat{P_1},\ldots,\widehat{P_m}$ on $\widehat{C_i}$ and integers $a_1,\ldots,a_m$ such that $\sum a_\ell[k(\widehat{P_\ell}):k]=\ind(C_i)$. Applying Proposition~\ref{lem:shafarevich} to choose a linearly equivalent divisor if necessary, we can ensure that none of the $\widehat{P_\ell}$ map to singularities in $\mathfrak{X}_k$. Let $P_\ell$ denote the image of $\widehat{P_\ell}$ in $C_i$. 
    
    Since $R$ is henselian, the reduction map from closed points of $X$ to closed points of $\mathfrak{X}_k$ is surjective~\cite[Corollary 10.1.38]{liu}. That is, for each $P_\ell$, there exists a closed point $Q_\ell$ of $X$ such that $\overline{\{Q_\ell\}}\cap \mathfrak{X}_k=\{P_\ell\}$. Define the divisor $D:=\sum a_j[Q_j]$ on $X$. Since the Zariski closure of each $Q_\ell$ only intersects $\mathfrak{X}_k$ at $P_\ell$, we have $C_j\cdot \overline{\{Q_\ell\}}=0$ for all $j\neq i$, and $C_i\cdot \overline{\{Q_\ell\}}=[k(P_\ell):k]$. This implies $\rho(D)=\ind(C_i)[C_i]$ as desired.
\end{proof}

\section{The Arithmetic Component Group}\label{sec:arithcomponent}

\subsection{An Arithmetic Version of Raynaud's Theorem}

We introduce some definitions. Let
\[\Div^{(0)}(X)=\Div(X)\cap \ker\rho,\qquad \Prin^{(0)}(X)=\Prin(X)\cap\ker\rho,\]
and $\Pic^{(0)}(X)$ be the image of $\Div^{(0)}(X)$ in $\Pic(X)=\Div(X)/\Prin(X)$.

\begin{defn}
The \emph{arithmetic component group of $\Pic^0(X)$} is defined to be $\Phi^a_X:=\Pic^0(X)/\Pic^{(0)}(X)$.
\end{defn}

The reason for the name is that if $k$ is algebraically closed, then $\Phi^a_X$ can be identified with the (geometric) component group $\Phi_X$ of the special fiber of the N\'eron model of the Jacobian of $X$~\cite[Appendix A]{matt}. More generally (under certain mild conditions), we will have an injection $\Phi^a_X\hookrightarrow\Phi_X$ (Proposition~\ref{prop:arithgeocomp}), but the two groups may be different; this reflects the fact that divisors defined over $R$ are too coarse to distinguish individual geometric components of $\mathfrak X_k$. We discuss the relationship between these groups in Section~\ref{sec:arithgeocomp}.

The arithmetic component group can be computed using data about $G$, as described below. First, observe that we can assign weights to the vertices and edges of $G$: the weight of the vertex corresponding to a component $C_i$ will be $\ind(C_i)$, and the weight of an edge corresponding to a node $p$ will be $[k(p):k]$. Since the degree of any node on $C_i$ must be a multiple of the index, this is a pleasant weighting of $G$. This allows us to define the group of balanced divisors $\Div_b(G)$ as in Section~\ref{sec:balanceddivs} (divisors such that the coefficient of $[C_i]$ is a multiple of $\ind(C_i)$), and set $\Pic_b^0(G):=\Div_b^0(G)/\Prin(G)$.

\arithraynaud*

\begin{proof}
    We have $\Div_u(G):=\prod_{i=1}^\gamma \Z/\ind(C_i)\Z$, and there is a natural map $\Div(G)\to\Div_u(G)$ obtained by reducing the coefficient of $[C_i]$ modulo $\ind(C_i)$ for every $i$. We can also define a map $\alpha:\Div_u(G)\to \Z/\ind(\mathfrak{X}_k)\Z$ by taking the entries at each component and summing their reductions mod $\ind(\mathfrak{X}_k):=\gcd\{\ind(C_i)\}$, and let $\Div^0_u(G)$ be the kernel of this map. This setup gives us the following commutative diagram:
    \[\begin{tikzcd}
                & 0 \arrow[d]  & 0 \arrow[d]   & 0 \arrow[d]   & 0 \arrow[d]             \\
                0\arrow[r] &\Div^{(0)}(X) \arrow[r] \arrow[d] & \Div^0(X) \arrow[d] \arrow[r, "\rho"]       &
                \Div^0(G)\arrow[d]  \arrow[r]       &\Div^0_u(G) \arrow[r] \arrow[d]           & 0\\
                0\arrow[r] &\Div^{(0)}(X) \arrow[r] \arrow[d] & \Div(X) \arrow[d, "\deg"] \arrow[r, "\rho"]       &
                \Div(G)\arrow[d, "\deg"]  \arrow[r]       &\Div_u(G) \arrow[r] \arrow[d, "\alpha"]           & 0\\
                & 0\arrow[r] &\ind(\mathfrak{X}_k)\Z \arrow[r] \arrow[d] & \Z \arrow[d] \arrow[r]       &
                \Z/\ind(\mathfrak{X}_k)\Z\arrow[d]  \arrow[r]           & 0\\
                & & 0 & 0 & 0.
    \end{tikzcd}\]
    The middle row is exact by Lemma~\ref{lem:specialdiv}, and the columns are obtained by taking kernels and images of specified maps, and are hence also exact. The top-right square commutes because $\Div^0(G)$ lands in the kernel of $\alpha$ under the map $\Div(G)\to\Div_u(G)$. The bottom-right square commutes because the composition in either direction is obtained by summing the coefficients of a divisor modulo $\ind(\mathfrak{X}_k)$. Since the second and third row are exact, so is the first.
    
    We now build another commutative diagram\footnote{In the case $k=\overline{k}$, this diagram reduces to~\cite[(A.6)]{matt}.} by the inclusion of principal divisors:
    \[\begin{tikzcd}
                & 0 \arrow[d]  & 0 \arrow[d]   & 0 \arrow[d]   &            \\
                0\arrow[r] &\Prin^{(0)}(X) \arrow[r] \arrow[d] & \Prin(X) \arrow[d] \arrow[r, "\rho"]       &
                \Prin(G)\arrow[d]  \arrow[r]       & 0\arrow[d]\\
                0\arrow[r] &\Div^{(0)}(X) \arrow[r] \arrow[d] & \Div^0(X) \arrow[d] \arrow[r, "\rho"]       &
                \Div^0(G)\arrow[d]  \arrow[r]       &\Div^0_u(G) \arrow[r] \arrow[d]           & 0\\
                 0\arrow[r] &\Pic^{(0)}(X) \arrow[r] \arrow[d] & \Pic^0(X) \arrow[d] \arrow[r]       &
                \Pic^0(G) \arrow[d]  \arrow[r]           & \Div_u^0(G) \arrow[d]  \arrow[r]           & 0\\
                & 0 & 0 & 0 & 0.
    \end{tikzcd}\]
    The columns are exact by definition of $\Pic^{(0)}(X)$, $\Pic^0(X)$, and $\Pic^0(G)$, and the first row is exact by Lemma~\ref{lem:specialprin}. Since we proved that the middle row is exact, we can conclude that the bottom row is also exact. We obtain the desired result by observing that
    \[\ker(\Pic^0(G)\to \Div_u^0(G))=\ker(\Div^0(G)\to \Div_u^0(G))/\Prin(G)=\Div_b^0(G)/\Prin(G).\qedhere\]
\end{proof}

We will now combine our arithmetic variants of Raynaud's Theorem~\ref{thm:arithraynaud} and the matrix-tree theorem~\ref{thm:countgraphcosets}. We assume that $\mathfrak X_k$ is connected so that the dual graph $G$ is connected. Choose a ``root'' component $C_1$ of $\mathfrak{X}_k$. To each balanced sub-weighted spanning tree $(T,\sigma)$ of $G$, we can define a divisor $D_{T,\sigma}$; since it is balanced, there will exist $\widetilde{D_{T,\sigma}}\in\Div(X)$ with $\rho(\widetilde{D_{T,\sigma}})=D_{T,\sigma}$ (Lemma~\ref{lem:specialdiv}). We can now identify the set of spanning trees with balanced edge sub-weightings as a torsor over the arithmetic component group, as follows.

\begin{cor}\label{cor:spantreedecomp}
    Assume $R$ is henselian, and that $\mathfrak X_k$ is connected. The set 
    \[\{\widetilde{D_{T,\sigma}}:\text{$T$ a spanning tree for $G$, $\sigma$ a balanced edge sub-weighting for $T$}\}\]
    forms a complete set of representatives for the cosets of $\Pic^{(0)}(X)$ that are contained in $\Pic^{g-1}(X)$:
    \[\Pic^{g-1}(X)=\bigsqcup_{(T,\sigma)} \widetilde{D_{T,\sigma}}+\Pic^{(0)}(X).\]
    Furthermore, the set of pairs $(T,\sigma)$, where $T$ is a spanning tree of $G$ and $\sigma$ is a balanced edge sub-weighting for $T$, is a torsor over $\Phi^a_X$.
\end{cor}

\begin{proof}
    As a consequence of Proposition~\ref{thm:arithraynaud}, the specialization map $\rho:\Div(X)\to \Div_b(G)$ induces an isomorphism
    \[\Pic^0(X)/\Pic^{(0)}(X)\cong \Div_b^0(G)/\Prin(G).\]
    Since $\Div_b^{g-1}(G)$ is a torsor over $\Div_b^0(G)$ and $\Pic^{g-1}(X)$ is a torsor over $\Pic^0(X)$, the result follows by Theorem~\ref{thm:countgraphcosets}.
\end{proof}

\begin{rem}
    Corollary~\ref{cor:spantreedecomp} can be made explicit. That is, given a a spanning tree $T$ with a balanced edge sub-weighting $\sigma$, and a divisor $D\in \Pic^0(X)$, we can compute how $D$ acts on $(T,\sigma)$ as follows. After computing the divisor $R:=\rho(D)+D_{T,\sigma}\in \Div^{g-1}(G)$, it suffices to find the spanning tree $T'$ and edge sub-weighting $\sigma'$ such that $R\in D_{T',\sigma'}+\Prin(G)$; then the torsor action will be given by $D+(T,\sigma)=(T',\sigma')$.
    
    As in the proof of Theorem~\ref{thm:countgraphcosets}, we can replace $G$ with the unweighted graph $\widehat{G}$, which has first betti number $\widehat{g}$. Then $R+\sum_{v\in V}(\omega(v)-1)[v]$ is an element of $\Div^{\widehat{g}-1}(\widehat{G})$. We can apply chip-firing moves to this divisor until a $C_1$-orientable divisor is obtained; a simple exponential-time algorithm to do this is described in Remark 4.14 of~\cite{ABKS}, and a more intricate but polynomial-time algorithm is outlined in~\cite{spencer}. The resulting divisor will be of the form $D_{\widehat{T'}}$ for some spanning tree $\widehat{T'}$ of $\widehat{G}$, which we can then use to construct a spanning tree $T'$ of $G$ and an edge sub-weighting $\sigma'$. Then 
    \begin{align*}
        D_{T',\sigma'}&=D_{\widehat{T'}}-\sum_{v\in V}(\omega(v)-1)[v]\\
        &\sim \left(R+\sum_{v\in V}(\omega(v)-1)[v]\right)-\sum_{v\in V}(\omega(v)-1)[v]\\
        &=\rho(D)+D_{T,\sigma},
    \end{align*}
    so this agrees with the torsor action as described in Corollary~\ref{cor:spantreedecomp}.
\end{rem}

\subsection{Comparing the Arithmetic and Geometric Component Groups}\label{sec:arithgeocomp}

So far, we have established bijections between $\Phi_X^a$, $\Pic_b^0(G)$, and the set of spanning trees of $G$ with balanced edge sub-weightings. It is worth considering how each of these is affected by base change. Let $R'/R$ be an extension of DVRs; for now, we will only consider the case that $\mathfrak X'=\mathfrak X_{R'}$ is a regular semistable model of its smooth generic fiber $X'$. Let $k'=R'\otimes_R k$, and let $G'$ be the dual graph of the special fiber $\mathfrak X_{k'}$. 

There are three ways that $G'$ may differ from $G$ as a weighted graph. First, a weighted edge $e$ of $G$ may split into multiple parallel edges in $G'$, with the total weight preserved (this occurs if a node over $k$ splits over $k'$). This operation preserves both $\Prin(G)$ and $\Div_b^0(G)$, and as the proof of Theorem~\ref{thm:countgraphcosets} indicates, this does not change the size of the set of spanning trees with balanced edge sub-weightings. Hence $\Phi_{X'}^a\cong \Phi_X^a$.

Second, the index of a component may shrink, so that $G'$ is isomorphic to $G$ as a graph, but has a lower weight at some vertex and hence more balanced divisors (this can occur because there are more line bundles defined over $k'$ than over $k$). Third, a vertex of $G$ may split into multiple vertices in $G'$ (this occurs if a component of $\mathfrak X_k$ splits into multiple components over $k'$). We will show that in both of these cases, $\Phi_{X}^a$ injects into $\Phi_{X'}^a$.

Note that we have an injection $\Psi:\Div(\mathfrak X)\to \Div(\mathfrak X')$, which can be defined on prime divisors by base change to $R'$ and extended by linearity. By restricting to the set of horizontal divisors and to the degree $0$ part, we get a map $\Div^0(X)\to \Div^0(X')$. 

\begin{prop}
    The map $\Psi$ induces an injection $\Phi_X^a\to \Phi_{X'}^a$.
\end{prop}

\begin{proof}
    By Theorem~\ref{thm:arithraynaud}, the specialization map $\rho:\Div^0(\mathfrak X)\to\Div_b^0(G)$ induces an isomorphism
    \[\Phi^a_X=\Div^0(X)/(\Prin(X)+\Div^{(0)}(X))\xrightarrow{\cong} \Div_b^0(G)/\Prin(G),\]
    and we have an analogous result for the specialization map $\rho':\Div^0(\mathfrak X')\to\Div_b^0(G')$.
    
    Now suppose an irreducible component $C_i$ of $\mathfrak X_k$ splits into irreducible components $C_i^{(1)},\ldots,C_i^{(r_i)}$ over $k'$. These $k'$-components are permuted transitively by $\text{Gal}(k'/k)$, and since $\Psi(\mathcal{D})$ is preserved by the Galois action, the value of $\Psi(\mathcal{D})\cdot C_i^{(j)}$ must be invariant for all $j=1,\ldots r$. This proves that $\ind(C_i)$ is a multiple of $r_i$, and that a divisor on $\mathfrak X$ with degree $r_id$ on $C_i$ will have degree $d$ on each $C_i^{(j)}$. In particular, if $\mathcal{D}\in\ker\rho$, then $(\rho'\circ \Psi)(\mathcal{D})=0$, which implies that $\Psi$ descends to a map $\psi$:
    \[\begin{tikzcd}
    \Div^0(X)\arrow[hookrightarrow]{r}\arrow[d, "\Psi"] & \Div^0(\mathfrak X)\arrow[r, "\rho"]\arrow[d, "\Psi"] & \Div^0_b(G)\arrow[d, "\psi"] \\
    \Div^0(X')\arrow[hookrightarrow]{r} & \Div^0(\mathfrak X')\arrow[r, "\rho'"]& \Div^0_b(G')\\
    \end{tikzcd}
    \]
    where $\psi$ is defined by
    \[\psi:\ind(C_i)[C_i]\mapsto \sum_{j=1}^{r_i} \frac{\ind(C_i)}{r_i}[C_i^{(j)}].\]
    It now suffices to prove that $\psi:\Div^0_b(G)\to \Div^0_b(G')$ descends to an injection $\Div^0_b(G)/\Prin(G)\to\Div^0_b(G')/\Prin(G')$. 
    
    To show this map is well-defined, consider an element of $\Prin(G)$.
    It was observed in the proof of Lemma~\ref{lem:specialprin} that the set of vertical divisors on $\mathfrak X$ map surjectively onto $\Prin(G)$, so we can write this element as $\rho(\mathcal V)$ for some vertical divisor $\mathcal V\in\Div(\mathfrak X)$. Then $\psi(\rho(\mathcal V))=\rho'(\Psi(\mathcal V))$, which is in $\Prin(G')$ because the base change of a vertical divisor is vertical.
    
    To show the map is injective, suppose $D\in\Div_b^0(G)$ satisfies $\psi(D)\in\Prin(G')$. Choosing $\mathcal{D}\in\Div(\mathfrak X)$ mapping to $D$ via $\rho$, we have $\rho'(\Psi(\mathcal{D}))=\psi(\rho(\mathcal{D}))\in\Prin(G')$, which implies that $\Psi(\mathcal{D})$ is a vertical divisor. Since a divisor is vertical if and only if its base change is, we can conclude that $\mathcal{D}$ is a vertical divisor, and hence that $D=\rho(\mathcal{D})$ is in $\Prin(G)$.
\end{proof}

Thus we have established that $\Phi_X^a$ embeds naturally into $\Phi_{X'}^a$, assuming that there is a semistable model of $X'$ which is the base change of a semistable model of $X$. We now show that $\Phi_X^a$ embeds naturally into the component group $\Phi_X$ of the special fiber of the N\'eron model of the Jacobian of $X$, and provide a condition under which the two are isomorphic.

\begin{prop}\label{prop:arithgeocomp}
    Assume $R$ is henselian and $X$ is geometrically irreducible. Suppose also that either $k$ is perfect or that $\mathfrak X$ admits an \'etale quasi-section. Then there is an injective homomorphism $\Phi_X^a\to \Phi_X$. If every irreducible component of $\mathfrak{X}_k$ has index $1$, then this injection is an isomorphism.
\end{prop}
\begin{rem}
    The technical condition requiring $k$ to be perfect, or $\mathfrak X$ to admit an \'etale quasi-section, is just for the existence of a N\'eron model of a particular form. As discussed in~\cite[Remark 9.5/5]{neron}, the assumption that $\mathfrak X$ admits an \'etale quasi-section is satisfied if any irreducible component of $\mathfrak X_k$ is geometrically reduced (note that $\mathfrak X$ is semi-stable, so irreducible components are necessarily reduced). 
    
    Proposition~\ref{prop:arithgeocomp} is originally proved in~\cite{boschliu}.  We provide a proof here as well, which contains many of the same ideas but uses different language. For definitions and background on the material used in the following proof, see~\cite[Chapter 9]{neron}.
\end{rem}

\begin{proof}
    Let $\PPic_{\mathfrak{X}/R}$ denote the relative Picard functor of $\mathfrak{X}$ over $R$, let $P$ be the open subfunctor of $\PPic_{\mathfrak{X}/R}$ given by line bundles of total degree $0$, and let $P^\circ$ denote the connected component of $\PPic_{\mathfrak{X}/R}$ containing the identity (since total degree is locally constant, this equals the connected component of $P$ containing the identity). If we let $Q$ be the largest separated quotient of $P$, then $Q$ is a N\'eron model of the Jacobian of $X$~\cite[Theorem 9.5/4]{neron} and the projection $P\to Q$ induces an isomorphism $P_K\to Q_K$ of generic fibers~\cite[Proposition 9.5/3]{neron}. Further, the canonical map $P(R)\to P(K)$ is surjective and the map $Q(R)\to Q(K)$ is bijective~\cite[p.~268]{neron}, so we have the following commutative diagram:
    \[\begin{tikzcd}
          P^\circ(R)\arrow[twoheadrightarrow]{d}\arrow[hookrightarrow]{r}  & P(R)\arrow[twoheadrightarrow]{d}\arrow[twoheadrightarrow]{r} & P(K) \arrow[d, "\cong"]\\
          Q^\circ(R)\arrow[hookrightarrow]{r}  & Q(R)\arrow[r, "\cong"] & Q(K).
    \end{tikzcd}\]
    
    By definition, we can identify $\Pic^0(X)$ with $P(K)$. The subfunctor $P^\circ$ consists of all elements of $P$ whose partial degree on each irreducible component of $\mathfrak{X}_k\otimes_{k} \overline{k}$ is zero~\cite[Corollary 9.3/13]{neron}. 
    If an irreducible component $C_i$ of $\mathfrak{X}_k$ splits into geometric components $C_i^{(j)}$ in $\mathfrak{X}_k\otimes_{k} \overline{k}$, then $\text{Gal}(\overline{k}/k)$ permutes the geometric components transitively~\cite[04KZ]{stacks} and so any line bundle defined over $R$ must have the same partial degree on each $C_i^{(j)}$. Hence $P^\circ(R)$ consists of all elements of $P(R)$ whose partial degree on each irreducible component of $\mathfrak{X}_k$ is zero; that is, $P^\circ(R)=(\ker\rho)/\Prin(\mathfrak X)$. Hence the image of $P^\circ(R)$ under the restriction map $P(R)\to P(K)$ is
    \[(\Div(X)\cap\ker\rho)/(\Prin(X)\cap\ker\rho)= \Pic^{(0)}(X).\]
    So by following the commutative diagram, we can conclude that
    \[\Phi_X^a=\Pic^0(X)/\Pic^{(0)}(X)= P(K)/\text{im}(P^\circ(R))\cong Q(R)/Q^\circ(R).\]
    The map $Q(R)\to Q(\overline{k})\to Q(\overline{k})/Q^\circ(\overline{k})$ has kernel $Q^\circ(R)$, so $\Phi_X^a$ injects into $Q(\overline{k})/Q^\circ(\overline{k})$. Since $Q_k(\overline{k})\cong Q(\overline{k})$ is dense in $Q_k$, $Q(\overline{k})/Q^\circ(\overline{k})$ is isomorphic to the component group $\Phi_X$ of $Q_k$, proving that $\Phi^a_X$ injects into $\Phi_X$.
    
    Now assume that every component has index $1$, and consider the map
    \[P(R)\to P(\overline{k}) \xrightarrow{\pi} Q(\overline{k}).\]
    Given any element of $Q(\overline{k})$, it can be written in the form $\pi(\ell)$ for some $\ell\in P(\overline{k})$. Consider the degree of $\ell$ on each component of $\mathfrak X_k$. By Lemma~\ref{lem:specialdiv}, there exists an element of $P(R)$ mapping to some $\ell'\in P(\overline{k})$ that has the same degree as $\ell$ on each component. This proves that $\ell-\ell'\in P^\circ(\overline{k})$, and hence $\pi(\ell)-\pi(\ell')\in Q^\circ(\overline{k})$, so the map $P(R)\to Q(\overline{k})/Q^\circ(\overline{k})$ is a surjection. Since this map factors through $Q(R)$, the injection 
    \[\Phi^a_X\cong Q(R)/Q^\circ(R)\to Q(\overline{k})/Q^\circ(\overline{k})\cong \Phi_X\]
    is also surjective.
\end{proof}

This allows us to prove the variant of Raynaud's Theorem stated in the introduction, copied below. Recall that we do not assume $R$ is henselian, and $k$ need not be algebraically closed.

\raynaud*

\begin{proof}
    First consider the case that $R$ is henselian. By assuming all components $C_i$ of $\mathfrak{X}_k$ have a smooth $k$-point, we guarantee that $\ind(C_i)=1$. Hence $\Div_u^0(G)=0$, and so by Theorem~\ref{thm:arithraynaud}, we have $\Phi^a_X\cong\Pic^0(G)$. Since index $1$ implies geometrically reduced~\cite[Corollary 9.1/8]{neron}, we can apply Proposition~\ref{prop:arithgeocomp} to conclude that $\Phi_X\cong\Pic^0(G)$.
    
    In the case that $R$ is not henselian, let $R^{h}$ denote its henselization, let $\mathfrak{X}^{h}$ denote the base change of $\mathfrak{X}$ to $R^{h}$, and let $X^h$ denote the generic fibre of $\mathfrak{X}^h$. The residue fields of $R$ and $R^{h}$ are the same, so $(\mathfrak{X}^h)_k\cong\mathfrak{X}_k$ and the dual graph is unaffected by this base change. The extension $K^h/K$ is an \'etale field extension since it is separable. The formation of N\'eron models commutes with \'etale base change, so letting $Q$ denote the N\'eron model of the Jacobian of $X$, $Q_{R^h}$ is a N\'eron model for the Jacobian of $\mathfrak{X}^{h}$. But then $(Q_{R^h})_k\cong Q_k$, and so their component groups $\Phi_{X^h}$ and $\Phi_X$ are isomorphic. Applying the result in the henselian case, we have $\Phi_X\cong \Phi_{X^h}\cong \Pic^0(G)$.
\end{proof}

\end{document}